\documentclass[11pt]{article}
\usepackage{authblk}
\usepackage{amssymb}
\usepackage[all]{xy}
\usepackage{mathrsfs}
\usepackage{amsmath,amssymb,amsthm}
\usepackage{extarrows}
\usepackage{mathrsfs}
\newtheorem{theorem}{Theorem}[section]

\def\to{\rightarrow}

\def\c{\mathcal}
\def\b{\mathbf}
\def\r{\mathrm}

\date{}
\begin{document}
%\ttfamily
%\sffamily
%\bfseries
%%%%%%%%%%%%%%%%%%%%%%%%%%%%%%%%%%%%%%%%%%%%%%%%%%%%%%%%%%%%%%%%%%%%%%%%%%%%%%%%%%%%%%%%%%%%%%%%%%%%%%%%%%%%%%%%%%%%%%%%%%
\title{A note on the best approximation in spaces of affine functions}
%%%%%%%%%%%%%%%%%%%%%%%%%%%%%%%%%%%%%%%%%%%%%%%%%%%%%%%%%%%%%%%%%%%%%%%%%%%%%%%%%%%%%%%%%%%%%%%%%%%%%%%%%%%%%%%%%%%%%%%%%%
\author{Maysam Maysami Sadr\thanks{sadr@iasbs.ac.ir}}
\affil{Department of Mathematics, Institute for Advanced Studies in Basic Sciences, Zanjan, Iran}
%%%%%%%%%%%%%%%%%%%%%%%%%%%%%%%%%%%%%%%%%%%%%%%%%%%%%%%%%%%%%%%%%%%%%%%%%%%%%%%%%%%%%%%%%%%%%%%%%%%%%%%%%%%%%%%%%%%%%%%%%%
\maketitle
%%%%%%%%%%%%%%%%%%%%%%%%%%%%%%%%%%%%%%%%%%%%%%%%%%%%%%%%%%%%%%%%%%%%%%%%%%%%%%%%%%%%%%%%%%%%%%%%%%%%%%%%%%%%%%%%%%%%%%%%%%
\begin{abstract}
The proximinality of certain subspaces of spaces of bounded affine functions is proved.
The results presented here are some linear versions of an old result due to Mazur.
For the proofs we use some sandwich theorems of Fenchel's duality theory.

\textbf{MSC 2010.} Primary 41A65; Secondary 52A07

\textbf{Keywords.} Best approximation, convex set, affine function, sandwich theorem
\end{abstract}
%%%%%%%%%%%%%%%%%%%%%%%%%%%%%%%%%%%%%%%%%%%%%%%%%%%%%%%%%%%%%%%%%%%%%%%%%%%%%%%%%%%%%%%%%%%%%%%%%%%%%%%%%%%%%%%%%%%%%%%%%%
%%%%%%%%%%%%%%%%%%%%%%%%%%%%%%%%%%%%%%%%%%%%%%%%%%%%%%%%%%%%%%%%%%%%%%%%%%%%%%%%%%%%%%%%%%%%%%%%%%%%%%%%%%%%%%%%%%%%%%%%%%
%%%%%%%%%%%%%%%%%%%%%%%%%%%%%%%%%%%%%%%%%%%%%%%%%%%%%%%%%%%%%%%%%%%%%%%%%%%%%%%%%%%%%%%%%%%%%%%%%%%%%%%%%%%%%%%%%%%%%%%%%%
\section{Introduction}
Let $S$ be a metric space and $T$ be a subspace of $S$. Then $T$ is called proximinal in $S$ if every element of $S$
has a best approximation by elements of $T$ i.e. for every $s\in S$ there exists $t_s\in T$ such that
$$d(s,t_s)=d(s,T):=\inf_{t\in T} d(s,t).$$
The proximinality problem for linear subspaces of normed (function) spaces \cite{Singer1} has been considered by many authors.
One of the early results in this direction is due to Mazur:
\begin{theorem}
Let $X,Y$ be compact Hausdorff spaces and $\phi:X\to Y$ be a surjective continuous map.
Let $\b{C}(X),\b{C}(Y)$ denote the Banach spaces of real valued continuous functions on $X,Y$ with supremum norm.
Then the image of $\b{C}(Y)$ in $\b{C}(X)$ under the canonical linear map induced by $\phi$, is proximinal.
\end{theorem}
Mazur's proof can be found in the Monograph of Semadeni \cite{Semadeni1} page 124. It uses the Hahn-Tong Sandwich Theorem
\cite[Theorem 6.4.4]{Semadeni1} for existence of continuous functions between upper and lower semicontinuous real valued functions.
The Mazur result have been extended for spaces of complex valued functions by Pe{\l}czy\'{n}ski \cite{Pelczynski1} and for vector valued
functions by Olech \cite{Olech1} and Blatter \cite{Blatter1}. There is also a generalization for vector spaces of `continuous functions'
on `noncommutative spaces' in terms of C*-algebras \cite{Somerset1}. For more details see \cite[page 15]{Singer1}.

In this short note we show that some analogs of the Mazur result are satisfied for spaces of bounded affine functions.
Our proofs are linear versions of the Mazur proof. But instead of the Hahn-Tong Theorem we use some sandwich theorems of Fenchel's duality theory
for existence of affine functions between concave and convex functions.
%%%%%%%%%%%%%%%%%%%%%%%%%%%%%%%%%%%%%%%%%%%%%%%%%%%%%%%%%%%%%%%%%%%%%%%%%%%%%%%%%%%%%%%%%%%%%%%%%%%%%%%%%%%%%%%%%%%%%%%%%%
%%%%%%%%%%%%%%%%%%%%%%%%%%%%%%%%%%%%%%%%%%%%%%%%%%%%%%%%%%%%%%%%%%%%%%%%%%%%%%%%%%%%%%%%%%%%%%%%%%%%%%%%%%%%%%%%%%%%%%%%%%
%%%%%%%%%%%%%%%%%%%%%%%%%%%%%%%%%%%%%%%%%%%%%%%%%%%%%%%%%%%%%%%%%%%%%%%%%%%%%%%%%%%%%%%%%%%%%%%%%%%%%%%%%%%%%%%%%%%%%%%%%%
%%%%%%%%%%%%%%%%%%%%%%%%%%%%%%%%%%%%%%%%%%%%%%%%%%%%%%%%%%%%%%%%%%%%%%%%%%%%%%%%%%%%%%%%%%%%%%%%%%%%%%%%%%%%%%%%%%%%%%%%%%
%%%%%%%%%%%%%%%%%%%%%%%%%%%%%%%%%%%%%%%%%%%%%%%%%%%%%%%%%%%%%%%%%%%%%%%%%%%%%%%%%%%%%%%%%%%%%%%%%%%%%%%%%%%%%%%%%%%%%%%%%%
%%%%%%%%%%%%%%%%%%%%%%%%%%%%%%%%%%%%%%%%%%%%%%%%%%%%%%%%%%%%%%%%%%%%%%%%%%%%%%%%%%%%%%%%%%%%%%%%%%%%%%%%%%%%%%%%%%%%%%%%%%
\section{The Results}
For a nonempty convex set $C$ we denote by $\c{A}_\r{b}(C)$ the normed space of all bounded affine real valued functions on
$C$ with supremum norm. If $C$ is a compact convex subset of a topological vector space we denote by $\c{A}_\r{c}(C)$ the closed subspace
of $\c{A}_\r{b}(C)$ containing all continuous affine functions. Let $C,D$ be two (compact) convex sets and $\phi:C\to D$ be a
(continuous) surjective affine map. Then $\phi$ induces an isometric linear isomorphism
$\tilde{\phi}$ from ($\c{A}_\r{c}(D)$) $\c{A}_\r{b}(D)$ into ($\c{A}_\r{c}(C)$) $\c{A}_\r{b}(C)$ defined by
$\tilde{\phi}(h):=h\circ \phi$. In what follows we identify $\c{A}_\r{b}(D)$ (resp. $\c{A}_\r{c}(D)$) as a closed subspace of
$\c{A}_\r{b}(C)$ (resp. $\c{A}_\r{c}(C)$).
\begin{theorem}\label{1711101643}
Let $C,D$ be two convex sets and $\phi:C\to D$ be a surjective affine map. Then $\c{A}_\r{b}(D)$, as a subspace of $\c{A}_\r{b}(C)$
via $\tilde{\phi}$, is  proximinal.
\end{theorem}
\begin{proof}
Suppose that $f\in\c{A}_\r{b}(C)$. We must show that there exists $h_0\in\c{A}_\r{b}(D)$ such that
$\|f-h_0\circ\phi\|=d$ where
$$d:=\inf\{\|f-h\circ\phi\|:h\in\c{A}_\r{b}(D)\}.$$
Let $c:=\sup_{y\in D} r(y)$ where
$$r(y):=\sup_{x,x'\in\phi^{-1}(y)}(f(x)-f(x')).$$
For $h\in\c{A}_\r{b}(D)$ and $x,x'\in\phi^{-1}(y)$ we have,
$$|f(x)-f(x')|\leq|f(x)-h\circ\phi(x)|+|f(x')-h\circ\phi(x')|.$$
This shows that
\begin{equation}\label{1711101642-1}
\|f-h\circ\phi\|\geq2^{-1}c
\end{equation}
and
\begin{equation}\label{1711101642-2}
d\geq2^{-1}c.
\end{equation}
Let $f^\downarrow,f^\uparrow$ be bounded real valued functions
on $D$ defined by
$$f^\downarrow(y):=\inf_{x\in\phi^{-1}(y)}f(x)$$
and
$$f^\uparrow(y):=\sup_{x\in\phi^{-1}(y)}f(x).$$
Then it is easily verified
that $f^\downarrow$ is convex and $f^\uparrow$ is concave. Also,
\begin{equation}\label{1711101644-1}
r(y)=f^\uparrow(y)-f^\downarrow(y)
\end{equation}
and
\begin{equation}\label{1711101644-2}
f^\uparrow-2^{-1}c\leq f^\downarrow+2^{-1}c.
\end{equation}
By the Sandwich theorem \cite[Corollary 2.4.1]{BorweinVanderwerff1} there is an affine function $h_0:D\to\mathbb{R}$ with
\begin{equation}\label{1711101646}
f^\uparrow-2^{-1}c\leq h_0\leq f^\downarrow+2^{-1}c.
\end{equation}
Thus $h_0\in\c{A}_\r{b}(D)$ and $\|f-h_0\circ\phi\|\leq2^{-1}c$. Hence, by (\ref{1711101642-1}) and (\ref{1711101642-2}), $\|f-h_0\circ\phi\|=d$.
\end{proof}
A continuous version of Theorem \ref{1711101643} is as follows.
\begin{theorem}
Let $C$ be an arbitrary compact convex set and $D$ be a compact convex subset of a Fr\'{e}chet topological vector space.
Let $\phi:C\to D$ be a surjective continuous affine map. Then $\c{A}_\r{c}(D)$, as a subspace of $\c{A}_\r{c}(C)$ via
$\tilde{\phi}$, is  proximinal.
\end{theorem}
\begin{proof}
Suppose that $f\in\c{A}_\r{c}(C)$. Let
$$d:=\inf\{\|f-h\circ\phi\|:h\in\c{A}_\r{c}(D)\}.$$
Also, let $r(y),c,f^\uparrow,f^\downarrow$
be as in the proof of Theorem \ref{1711101643}. Thus, for every $h\in\c{A}_\r{c}(D)$, (\ref{1711101642-1}), (\ref{1711101642-2}),
(\ref{1711101644-1}), and (\ref{1711101644-2}) are satisfied.
It follows from \cite[Lemma 7.5.5]{Semadeni1} that $f^\downarrow$ and $f^\uparrow$ are respectively lower and upper semicontinuous functions.
By the Sandwich theorem \cite[Theorem 6(2)]{Noll1} of Noll, there exists a continuous affine function $h_0:D\to\mathbb{R}$ satisfying
(\ref{1711101646}). Thus $h_0\in\c{A}_\r{c}(D)$ and $\|f-h_0\circ\phi\|=d$. The proof is complete.
\end{proof}
%%%%%%%%%%%%%%%%%%%%%%%%%%%%%%%%%%%%%%%%%%%%%%%%%%%%%%%%%%%%%%%%%%%%%%%%%%%%%%%%%%%%%%%%%%%%%%%%%%%%%%%%%%%%%%%%%%%%%%%%%%
%%%%%%%%%%%%%%%%%%%%%%%%%%%%%%%%%%%%%%%%%%%%%%%%%%%%%%%%%%%%%%%%%%%%%%%%%%%%%%%%%%%%%%%%%%%%%%%%%%%%%%%%%%%%%%%%%%%%%%%%%%
%%%%%%%%%%%%%%%%%%%%%%%%%%%%%%%%%%%%%%%%%%%%%%%%%%%%%%%%%%%%%%%%%%%%%%%%%%%%%%%%%%%%%%%%%%%%%%%%%%%%%%%%%%%%%%%%%%%%%%%%%%
%%%%%%%%%%%%%%%%%%%%%%%%%%%%%%%%%%%%%%%%%%%%%%%%%%%%%%%%%%%%%%%%%%%%%%%%%%%%%%%%%%%%%%%%%%%%%%%%%%%%%%%%%%%%%%%%%%%%%%%%%%
%%%%%%%%%%%%%%%%%%%%%%%%%%%%%%%%%%%%%%%%%%%%%%%%%%%%%%%%%%%%%%%%%%%%%%%%%%%%%%%%%%%%%%%%%%%%%%%%%%%%%%%%%%%%%%%%%%%%%%%%%%

%%%%%%%%%%%%%%%%%%%%%%%%%%%%%%%%%%%%%%%%%%%%%%%%%%%%%%%%%%%%%%%%%%%%%%%%%%%%%%%%%%%%%%%%%%%%%%%%%%%%%%%%%%%%%%%%%%%%%%%%%%
%%%%%%%%%%%%%%%%%%%%%%%%%%%%%%%%%%%%%%%%%%%%%%%%%%%%%%%%%%%%%%%%%%%%%%%%%%%%%%%%%%%%%%%%%%%%%%%%%%%%%%%%%%%%%%%%%%%%%%%%%%
%%%%%%%%%%%%%%%%%%%%%%%%%%%%%%%%%%%%%%%%%%%%%%%%%%%%%%%%%%%%%%%%%%%%%%%%%%%%%%%%%%%%%%%%%%%%%%%%%%%%%%%%%%%%%%%%%%%%%%%%%%
%%%%%%%%%%%%%%%%%%%%%%%%%%%%%%%%%%%%%%%%%%%%%%%%%%%%%%%%%%%%%%%%%%%%%%%%%%%%%%%%%%%%%%%%%%%%%%%%%%%%%%%%%%%%%%%%%%%%%%%%%%
%%%%%%%%%%%%%%%%%%%%%%%%%%%%%%%%%%%%%%%%%%%%%%%%%%%%%%%%%%%%%%%%%%%%%%%%%%%%%%%%%%%%%%%%%%%%%%%%%%%%%%%%%%%%%%%%%%%%%%%%%%
\end{document}